\documentclass[10pt]{article}  
\usepackage{amsmath}
\usepackage{amssymb}
\usepackage{theorem}
\usepackage{euscript}
\usepackage{pstricks}

\newtheorem{theorem}{Theorem}[section]
\newtheorem{lemma}[theorem]{Lemma}

\newtheorem{remark}[theorem]{Remark}
\numberwithin{equation}{section}
\title{\sffamily On  Whitney type inequalities for local anisotropic polynomial approximation}
\author{Dinh Dung$^a$\footnote{Corresponding author. Email: dinhzung@gmail.com}, Tino Ullrich$^b$ \\
$^a$ Vietnam National University, Hanoi, Information Technology Institute \\
144, Xuan Thuy, Hanoi, Vietnam  \\
$^b$Hausdorff-Center for Mathematics, 53115 Bonn, Germany}
\date{\ttfamily September 25, 2010 -- Version 1.6}
\def\ZZ{{\mathbb Z}}
\def\NN{{\mathbb N}}
\def\RR{{\mathbb R}}
\def\Pp{{\mathcal P}}

\newlength{\fixboxwidth}
\begin{document}
\maketitle

\begin{abstract}
We prove a multivariate Whitney type theorem for
the local anisotropic  polynomial approximation in $L_p(Q)$ with $1\leq p\leq \infty$.
Here $Q$ is a $d$-parallelepiped in $\RR^d$ with sides parallel to the coordinate axes. We consider the 
error of best approximation
of a function $f$ by algebraic polynomials of fixed degree at most $r_i - 1$ in variable $x_i,\ i=1,...,d$, 
and relate it to a 
so-called total mixed modulus of smoothness appropriate to characterizing 
the convergence rate of the approximation error. This theorem is derived from
 a Johnen type theorem on equivalence between a certain K-functional and the total mixed modulus of smoothness
which is proved in the present paper. 


\end{abstract}

\section{Introduction and main results}
The classical Whitney theorem establishes the equivalence between the modulus of 
smoothness $\omega_r(f,|I|)_{p,I}$ and the error of best approximation $E_r(f)_{p,I}$ of 
a function $f:I \to \mathbb{R}$ by algebraic polynomials of degree at most $r-1$ measured in 
$L_p$, $1\leq p\leq \infty$.
Namely, the following inequalities
\begin{equation} \label{ineq:WhitneyIneq[d=1]}
2^{-r} \omega_r(f,|I|)_{p,I} 
\ \le \ 
E_r(f)_{p,Q}  
\ \le \ 
C \omega_r(f,|I|)_{p,I}
\end{equation} 
hold true with a constant $C$ depending only on $r$. This result was first proved by Whitney 
\cite{Wh57} for $p=\infty$ and extended by Brudny\u{\i} \cite{Br64} to $1 \le p <\infty$. 
The inequalities \eqref{ineq:WhitneyIneq[d=1]} provide, in particular, a convergence 
characterization for a local polynomial approximation when the degree $r-1$ of polynomials is fixed and the 
interval $I$ is small.\\
\noindent
Several authors have dealt with this topic in order to extend and generalize 
the result in various directions.  
Let us briefly mention them. A multivariate (isotropic) generalization for functions on a coordinate 
$d$-cube $Q$ in $\RR^d$ was given by Brudny\u{\i} \cite{Br70} and \cite{Br70_2}. 
It turned out that the result keeps valid if one replaces the $d$-cube by a more general domain $\Omega$.
The case of a convex domain $\Omega \subset \RR^d$ is already treated in Brudny\u{\i} \cite{Br70}. Let us also 
refer to the recent contributions by Dekel, Leviatan \cite{De04} and Dekel \cite{De10} with focus on convex and Lipschitz domains and
the improvement of the constants involved. \\
A reasonable question is also to ask for the case $0<p<1$. 
We refer to the works of Storozhenko \cite{S}, Storozhenko, Oswald \cite{StOs78}, 
and in addition to the appendix of the substantial paper by Hedberg and Netrusov \cite{HeNe07} 
for a brief history and further references.\\
A natural question arises: Is there a Whitney type theorem for the anisotropic 
approximation of multivariate functions on a coordinate $d$-parallelepiped $Q$? Some work has been done in this 
direction, see for instance Garrig\'{o}s, Tabacco \cite{GaTo02}. However, the present paper deals with a 
rather different setting which is 
somehow related to the theory 
of function spaces with mixed smoothness properties \cite{Di86,ST87,Ull06,Vy06}. We intend to approximate 
a multivariate function $f$ by polynomials of fixed degree at most $r_i - 1$ in variable 
$x_i,\ i=1,...,d$, on a small $d$-parallelepiped $Q$.  
A total mixed modulus of smoothness is defined which turns out to be a suitable convergence characterization 
to this approximation.
The classical Whitney inequality can be derived as a corollary of Johnen's 
theorem \cite{J} on the equivalence of the $r$-th Peetre K-functional $K_r(f,t^r)_{p,I}$, see \cite{P}, and the 
modulus of smoothness $\omega_r(f,t)_{p,I}$. A proof was given by Johnen and Scherer in
\cite{JS76}. Following this approach to Whitney type theorems we will introduce the notion 
of a mixed K-functional and prove
its equivalence to the total mixed modulus of smoothness by generalizing the technique of Johnen and Scherer
to the multivariate mixed situation.
\subsection{Notation}
In order to give an exact setting of the problem and 
formulate the main results let us preliminarily introduce some necessary notations.
As usual, $\mathbb{N}$ is reserved for the natural numbers, by
$\mathbb{Z}$ we denote the set of all integers, and by $\RR$ the real numbers. 
Furthermore, $\ZZ_+$ and $\RR_+$ denote the set of non-negative integers and real numbers, respectively. 
Elements $x$ of $\RR^d$ will be denoted by $x = (x_1,...,x_d)$. 
For a vector $r\in \ZZ^d_+$ and $x \in \RR^d$, we will further write 
$$
    x^r := (x_1^{r_1},..., x_d^{r_d})\,.
$$
Moreover, if $x,y \in {\RR}^d$, the inequality $x \le y \ (x < y)$ means that $x_i \le y_i\ (x_i < y_i), \ i=1,...,d$. 
As usual, the notation $A \ll B$ indicates that there is a constant $c>0$ (independent of the parameters
which are relevant in the context) such that $A \le cB$, whereas
$A \asymp B$ is used if $A \ll B$ and $B \ll A$, respectively. 

If $r \in {\ZZ}^d_+$, let $\Pp_r$ be the set of algebraic polynomials of degree at most $r_i - 1$ 
at variable $x_i, \ i \in [d]$, where $[d]$ denotes the set of all natural numbers from $1$ to $d$. 
We intend to approximate a function
$f$ defined on a $d$-parallelepiped 
$$
      Q := [a_1,b_1]\times ... \times [a_d,b_d]
$$
by polynomials from the class $\mathcal{P}_r$. If $D \subset {\RR}^d$ is a domain in $\RR^d$ we denote by  
$L_p(D)$, $0<p\leq \infty$, the quasi-normed space 
of Lebesgue measurable functions on $D$ with the usual $p$-th integral quasi-norm 
$\|\cdot\|_{p,D}$ to be finite, whereas we use the ess sup norm if $p = \infty$. 
The error of the best approximation of $f \in L_p(Q)$ by polynomials from $\Pp_r$
is measured by
\begin{equation*}  
E_r(f)_{p,Q} := \inf_{\varphi \in \Pp_r} \ \|f -\varphi\|_{p,Q}. 
\end{equation*}
For $r \in {\ZZ}_+$, $h\in \RR$, and a univariate functions $f$, the 
$r$th difference operator $\Delta_h^r$ is defined by 
\begin{equation*}
\Delta_h^r(f,x) := \
\sum_{j =0}^r (-1)^{r - j} \binom{l}{j} f(x + jh)\quad,\quad \Delta_h^0f(x) := f(x)\,,
\end{equation*}
whereas for $r \in {\ZZ}^d_+$, $h\in \RR^d$ and a $d$-variate function $f:\RR^d \to \RR$,  
the mixed $r$th difference operator $\Delta_h^r$ is defined by 
\begin{equation*}
\Delta_h^r := \
\prod_{i = 1}^d \Delta_{h_i,i}^{r_i}\,.
\end{equation*}
Here, the univariate operator
$\Delta_{h_i,i}^{r_i}$ is applied to the univariate function $f$ by considering $f$ as a 
function of  variable $x_i$ with the other variables fixed. 
Let
\begin{equation*}
\omega_r(f,t)_{p, Q}:= \sup_{|h_i| \le t_i, i \in [d]} 
\|\Delta_h^r(f)\|_{p,Q_{rh}}\quad,\quad  t \in {\RR}^d_+,
\end{equation*} 
be the mixed $r$th modulus of smoothness of $f$,
where for $y,h \in \RR^d$, we write $yh:= (y_1h_1,...,y_dh_d)$ and
 $Q_y := \{x \in Q: x_i, x_i + y_i \in [a_i,b_i], \  i \in [d]\}$.
For $r \in {\ZZ}^d_+$ and $e \subset [d]$, 
denote by $r(e) \in {\ZZ}^d_+$ the vector with $r(e)_i = r_i, i \in e$ and 
$r(e)_i = 0, i \notin e$ \ ($r(\emptyset) = 0$).
We define the {\em total mixed modulus of smoothness of order $r$} by
\begin{equation*} 
\Omega_r(f,t)_{p,Q}
:= \ 
 \sum_{e \subset [d], e \ne \emptyset} \omega_{r(e)}(f,t)_{p,Q}, \ t \in {\RR}^d_+.
\end{equation*}
The total mixed modulus of smoothness is a modification of mixed moduli of smoothness 
which are an important tool for characterizing function spaces with mixed smoothness properties, 
see \cite{Di86,ST87} and the recent contributions \cite{Ull06,Vy06}.

\subsection{Main results}
In the present paper, we generalize the Whitney inequality \eqref{ineq:WhitneyIneq[d=1]}  
to the  error of the best local anisotropic approximation $E_r(f)_{p,Q}$ by polynomials from $\Pp_r$ and 
the total mixed modulus of smoothness $\Omega_r(f,t)_{p,Q}$. More precisely, we prove the 
following Whitney type inequalities.
 
\begin{theorem} \label{Theorem: Whitney}
Let $1 \le p \le \infty$, $r \in {\NN}^d$. Then there is a constant $C$ depending 
only on $r,d$ such that for every $f \in L_p(Q)$
\begin{equation} \label{ineq:WhitneyIneq} 
\Big(\sum_{e \subset [d]}\prod_{i \in e} 2^{r_i}\Big)^{-1} \Omega_r(f,\delta)_{p,Q}
\ \le \ 
E_r(f)_{p,Q}  
\ \le \ 
C \Omega_r(f,\delta)_{p,Q}, 
\end{equation}
where $\delta = \delta(Q):= (b_1-a_1,...,b_d-a_d)$ is the size of $Q$. 
\end{theorem}
Theorem \ref{Theorem: Whitney} shows that the total mixed modulus of smoothness $\Omega_r(f,t)_{p,Q}$
gives a sharp convergence characterization of the best anisotropic polynomial approximation
when $r$ is fixed and the size $\delta(Q)$ of the $d$-parallelepiped $Q$ is small. 
This may have applications in the approximation
of functions with mixed smoothness by piecewise polynomials or splines.

So far we focus on the case $1\leq p\leq \infty$. 
 This makes it possible to apply a technique 
developed by Johnen and Scherer \cite{JS76}. As mentioned above they showed the equivalence of Peetre's 
$K$-functional  of order $r$ with respect to a classical Sobolev space $W^r_p$ and 
the modulus of smoothness of order $r$ for the univariate case.
The question of a $K$-functional suitable for mixed Sobolev spaces has been often considered in the past. 
We refer for instance to Sparr \cite{Sp74} and DeVore et al. \cite{DePeTe94}.
By introducing a mixed $K$-functional $K_r(f,t)_{p,Q}, \ t \in {\RR}^d_+,$ 
(see the definition in Section \ref{Johnen's inequality}), such an equivalence between 
$K_r(f,t^r)_{p,Q}$ and the total mixed modulus of smoothness $\Omega_r(f,t)_{p,Q}$ can be established as well. 
Namely, we prove the following

\begin{theorem}  \label{Theorem: K_r><omega_r}
Let $1\le p \le \infty$ and $r \in {\NN}^d$. 
Then for any $f \in L_p(Q)$, the following inequalities
\begin{equation} \label{ineq: K_r><omega_r}
\Big(\sum_{e \subset [d]}\prod_{i \in e} 2^{r_i}\Big)^{-1}\Omega_r(f,t)_{p,Q} \ \le K_r(f,t^r)_{p,Q}
 \ \le \ 
C\Omega_r(f,t)_{p,Q}\quad,\quad t \in {\RR}^d_+, 
\end{equation}
hold true with a constants $C$  depending on $r,p,d$ only.
\end{theorem}
The paper is organized as follows. 
In Section \ref{Taylor} we establish an error formula for the anisotropic 
approximation by Taylor polynomials for functions from mixed Sobolev spaces. 
Section \ref{Johnen's inequality} is devoted 
to the equivalence of the total mixed modulus of smoothness and the mixed $K$-functional 
(Theorem \ref{Theorem: K_r><omega_r}) which 
is applied in Section \ref{Whitney's inequality} to derive the Whitney type inequality for 
the local anisotropic polynomial approximation (Theorem \ref{Theorem: Whitney}).

\section{Anisotropic approximation by Taylor polynomials} 
\label{Taylor}

By $f^{(k)}, \ k \in {\ZZ}^d_+$, we denote the $k$th order mixed weak derivative of $f$, i.e.,
$f^{(k)}:= \partial^{k_1 + \cdots + k_d} f/\partial x_1^{k_1} \cdots \partial x_d^{k_d}$
\ ($f^{(0)}= f$). 
For $r \in {\NN}^d$ and $1\leq p\leq \infty$, the Sobolev type space 
$W^r_p(Q)$ of mixed smoothness $r$ is defined as the set of functions 
$f \in L_p(Q)$, for which the following norm is finite
\begin{equation*}
\|f\|_{W^r_p(Q)}
:= \ 
 \sum_{e \subset [d]} \|f^{(r(e))}\|_{p,Q}.
\end{equation*}
If $f\in W^r_p(Q), k \le r$, $k,r \in {\ZZ}^d_+$, and $x^0 \in Q$, 
then we can define the Taylor polynomial $T_k(f)$ of order $k$ by
 \begin{equation*}  
T_k(f,x)  
\ := \ 
T_k(f,x^0,x)  
\ = \ 
\sum_{0 \le s < k} f^{(s)}(x^0)e_s(x-x^0) 
\end{equation*}
where $e_s(x) := \prod_{i=1}^d e_{s_i}(x_i)$ and $e_m(t) := t^m/m!$. 
We want to give an upper bound of the error of approximation of 
$f\in W^r_p(Q)$ by the Taylor polynomial $T_r$. For this purpose we need 
the following auxiliary lemma.

\begin{lemma} \label{Lemma: DerivativeIneq}
Let $1 \le p \le \infty$, $r\ge 1$ and $Q=[a,b]$. Then there exist constants $C_1,C_2$ depending 
only on $r$ such that for $k=0,...,r-1$ and $0 \le t \le b-a$ the inequalities 
\begin{equation} \label{ineq:DerivativeIneq} 
\begin{split}
t^k\|f^{(k)} \|_{p,Q}  
&\ \le \ 
C_1 (\|f\|_{p,Q} \ + \ t^r\|f^{(r)} \|_{p,Q})\,,\\
t^{k+1/p}\|f^{(k)} \|_{\infty,Q}  
&\ \le \ 
C_2 (\|f\|_{p,Q} \ + \ t^r\|f^{(r)} \|_{p,Q})
\end{split}
\end{equation} 
hold true for any $f\in W^r_p(Q)$\,. 
\end{lemma}

\begin{proof} {\em Step 1.} A proof of the first inequality can be found in \cite[page 38]{DL}. 
Let us prove the second one by using
the first one. We start with $r=1$ and $k=0$. We fix $y \in Q$ and consider a subinterval 
$Q_t$ of length $t$ containing $y$. We define the function 
$$
   g_t(x)=\frac{\chi_{Q_t}(x)}{t} \int_{Q_t}f(s)\,ds.
$$
 Since $f$ is continuous the function 
$f-g_t$ has a zero $\xi \in Q_t$. By Taylor's formula we have
$$
    (f-g_t)(y) = \int\limits_{\xi}^y (f'-g_t')(s)\,ds = \int\limits_{\xi}^y f'(s)\,ds\,.
$$
This gives $|f(y)|-|g_t(y)| \leq |f(y)-g_t(y)| \leq t^{1/p'}\|f'\|_p$ and hence
\begin{equation}\label{f1}
    |f(y)| \leq |g_t(y)|+t^{1/p'}\|f'\|_p \leq t^{1/p'-1}\|f\|_p+t^{1/p'}\|f'\|_p = t^{-1/p}(\|f\|_p+t\|f'\|_p)\,.
\end{equation}
Since we chose $y\in Q$ arbitrarily at the beginning we obtain \eqref{f1} for every $y\in Q$. This proves the second
inequality in  \eqref{ineq:DerivativeIneq} for $r=1$ and $k=0$.\\
{\em Step 2.} Let now $r>1$ and $k<r$. Taylor's formula gives
$$
    f^{(k)}(x) = \sum\limits_{\ell=0}^{r-1-k} f^{(k+\ell)}(\xi)\frac{(x-\xi)^{\ell}}{\ell!} 
    + \int\limits_{\xi}^x f^{(r)}(s) e_{r-k-1}(x-s)\,ds\,.
$$
Therefore, we have
\begin{equation*}
          |f^{(k)}(x)| \ll \sum\limits_{\ell=0}^{r-1-k} \|f^{k+\ell}\|_{\infty}t^{\ell} + \|f^{(r)}\|_pt^{r-k-1+1/p'}\,.
\end{equation*}
Applying the formula from Step 1 leads to 
\begin{equation*}
   \begin{split}       
      |f^{(k)}(x)| &\ll \sum\limits_{\ell=0}^{r-1-k} \|f^{k+\ell}\|_{\infty}t^{\ell} + \|f^{(r)}\|_pt^{r-k-1+1/p'}\\
      &\ll t^{-1/p} \sum\limits_{\ell=0}^{r-1-k} (\|f^{k+\ell}\|_{p}t^{\ell} 
      + \|f^{k+\ell+1}\|_{p}t^{\ell+1})+t^{r-k-1/p}\|f^{(r)}\|_p\\
      &= t^{-(k+1/p)} \Big(\|f^{(r)}\|_pt^r+\sum\limits_{\ell=0}^{r-1-k} (\|f^{k+\ell}\|_{p}t^{k+\ell} 
      + \|f^{k+\ell+1}\|_{p}t^{k+\ell+1})\Big)\,.
   \end{split}       
\end{equation*}
We apply the first inequality in \eqref{ineq:DerivativeIneq} and obtain for every $x$
$$
    t^{k+1/p}|f^{(k)}(x)| \ll \|f\|_p + t^r\|f^{(r)}\|_p\,
$$
which concludes the proof.
\end{proof}

\begin{theorem} \label{Theorem:UpperBnd[E_{r-1}(f)]}
Let $1 \le p \le \infty$, $r \in {\NN}^d$. Then there is a constant $C$ depending 
only on $r,d$ such that for every $f \in W^r_p(Q)$
\begin{equation*}  
E_r(f)_{p,Q}  
\ \le \ 
C \sum_{e \subset [d], e \ne \emptyset} \ \prod_{i \in e} \delta_i^{r_i} \|f^{(r(e))}\|_{p,Q}, 
\end{equation*}
where $\delta = \delta(Q)$.
 \end{theorem}

\begin{proof}
For simplicity we prove the theorem for the case $d=2$ and $Q=[0,b_1] \times [0,b_2]$.
If $g$ is a univariate function on the interval $[0,b]$ and $g \in W^k_p([0,b])$,
 then $g$ can be represented as
\begin{equation*}  
g(x) 
\ = \ 
T_k(g,0,x) \ + \ \int_0^x g^{(k)}(\xi)e_{k-1}(x - \xi) d\xi. 
\end{equation*}
Hence, for a function $f \in W^r_p(Q)$ we get
\begin{equation*}  
\begin{aligned}
f(x) 
\ & = \ 
T_r(f,x) \\
 \ & + \ 
\sum_{k_1=0}^{r_1-1} e_{k_1}(x_1) \int_0^{x_2} f^{((k_1,r_2))}(0,\xi_2)e_{r_2-1}(x_2 - \xi_2) d\xi_2 \\ 
\ & + \ 
\sum_{k_2=0}^{r_2-1} e_{k_2}(x_2) \int_0^{x_1} f^{((r_1,k_2))}(\xi_1,0)e_{r_1-1}(x_1 - \xi_1) d\xi_1 \\ 
\ & + \ 
 \int_0^{x_2} \int_0^{x_1} f^{(r)}(\xi_1,\xi_2)e_{r_1-1}(x_1 - \xi_1)e_{r_2-1}(x_2 - \xi_2) d\xi_1 d\xi_2,
\end{aligned}
\end{equation*}
where
\begin{equation*} 
T_r(f,x)
\ = \ 
\sum_{k_2=0}^{r_2-1}\sum_{k_1=0}^{r_1-1}f^{((k_1,k_2))}(0,0) e_{k_2}(x_2)e_{k_1}(x_1).
\end{equation*}
By triangle and H{\"o}lder's inequality we obtain
\begin{equation} \label{ineq:ErrorEstimation} 
\begin{aligned}
\|f - T_r(f)\|_{p,Q}  
\ & \ll \ 
\sum_{k_1=0}^{r_1-1} b_1^{k_1 + 1/p} b_2^{r_2}\|f^{((k_1,r_2))}(0,\cdot)\|_{p,[0,b_2]} \\
\ & + \ 
\sum_{k_2=0}^{r_2-1} b_1^{r_1} b_2^{k_2 + 1/p}\|f^{((r_1,k_2))}(\cdot,0)\|_{p,[0,b_1]}
\ + \ 
b_1^{r_1} b_2^{r_2}\|f^{(r)}\|_{p,Q}.
\end{aligned}
\end{equation} 
The following inequalities are a direct consequence of Lemma \ref{Lemma: DerivativeIneq}
\begin{equation}\label{eq1}  
\begin{split}
b_1^{k_1+1/p}\|f^{((k_1,r_2))}(0,\cdot)\|_{p,[0,b_2]}
&\ \ll \ 
\|f^{((0,r_2))}\|_{p,Q} \ + \ b_1^{r_1}\|f^{(r)}\|_{p,Q}\,, \\
b_2^{k_2+1/p}\|f^{((r_1,k_2))}(\cdot,0)\|_{p,[0,b_1]}
&\ \ll \ 
\|f^{((r_1,0))}\|_{p,Q} \ + \ b_2^{r_2}\|f^{(r)}\|_{p,Q}.  
\end{split}
\end{equation}
Indeed, put $I_j:= [0,b_j], \ j = 1,2$. If $f$ is in $W^r_p(Q)$,  
then almost all (with respect to $t \in I_1$) functions $g_t:=f^{(r_1,0)}(t,\cdot)$ belong to $W^{r_2}_p(I_2)$.
Now we can use Lemma \ref{Lemma: DerivativeIneq} to obtain for $k_2<r_2$
\begin{equation*}
b_2^{k_2+1/p}|g_t^{k_2}(0)| 
\leq b_2^{k_2+1/p}\|g_t^{k_2}\|_{\infty} 
\ll \|g_t\|_{p,I_2}+b_2^{r_2}\|g_t^{r_2}\|_{p,I_2}
\end{equation*}
for almost all $t$. The weak derivative $g_t^{k_2}$ coincides with 
$f^{(r_1,k_2)}(t,\cdot)$ in $L_p(I_2)$.
If $k_2 < r_2$, then $g_t^{k_2}$ and $f^{(r_1,k_2)}(t,\cdot)$ 
are continuous functions and the equation $g_t^{k_2}(0) = f^{(r_1,k_2)}(t,0)$ 
makes sense and holds true (we take the continuous representative) for almost all $t$.
Hence, using the estimate from above (the norm $\|\cdot\|_{p,I_1,t}$ 
on the right-hand side is with respect to $t$) we have
\begin{equation*}
\begin{split}
b_2^{k_2+1/p}\| f^{(r_1,k_2)}(\cdot,0)\|_{p,I_1}
\ & = \ b_2^{k_2+1/p}\| g_t(0) \|_{p,I_1,t} \\
\ & \ll \ \|\, \|g_t(\cdot)\|_{p,I_2} \, \|_{p,I_1,t}
  + b_2^{r_2}\| \, \|g_t^{r_2}(\cdot)\|_{p,I_2} \, \|_{p,I_1,t} \\
\ & = \ \|f^{(r_1,0)}\|_{p,Q} + b_2^{r_2}\| f^{(r_1,r_2)}\|_{p,Q}.
\end{split}
\end{equation*}
This proves the second line in \eqref{eq1}, the first one is obtained analogously.

Combining the inequalities \eqref{ineq:ErrorEstimation} and \eqref{eq1} proves the theorem. 
\end{proof}

\section{Johnen type inequalities for mixed K-functionals } 
\label{Johnen's inequality}

For $r \in {\NN}^d$, the {\em mixed K-functional $K_r(f,t)_{p,Q}$ } is defined for functions 
$f \in L_p(Q)$ and $t \in {\RR}^d_+$ by
\begin{equation*} 
K_r(f,t)_{p,Q}
:= \ 
 \inf_{g \in W^r_p(Q)} \{\|f - g\|_{p,Q} + 
\sum_{e \subset [d], e \ne \emptyset} \ \Big(\prod_{i \in e} t_i\Big) \|g^{(r(e))}\|_{p,Q}\}.
\end{equation*}
The following technical lemma needs a further notation. Let us assume $a_i \leq c_i<d_i\leq b_i$ for $i\in [d]$. 
We put
$I^i = [a_i,b_i]$, $I^i_1 = [a_i,d_i]$, and $I^i_0 = [c_i,b_i]$ and further
\begin{equation} \label{def:[Q_e]}
Q_e := \prod_{i=1}^d I^i_{\chi_e(i)}\,,
\end{equation}
where $\chi_e$ denotes the characteristic function of the set $e \subset [d]$\,.

\begin{lemma}  \label{Lemma: Ineq[K_r]}
Let $1\le p \le \infty$ and $r \in {\NN}^d$. 
Then for any $f \in L_p(Q)$  the inequality
\begin{equation*} 
  K_r(f,t^r)_{p,Q} \ \le \ C\sum_{e \subset [d]}  K_r(f,t^r)_{p,Q_e}
\end{equation*}
holds true for all $t \in {\RR}^d_+$ with $t_i \le d_i - c_i, \, i \in [d]$. The constant $C$ only depends on 
$r$ and $d$.
\end{lemma}

\begin{proof} The proof is based on an iterative argument. The first step is to observe
\begin{equation*}
  \begin{split}
    Q &= Q_1 \cup Q_0\\
     &= \Big(I^1_1 \times \prod\limits_{i\in [d]\setminus \{1\}} I^i\Big) \cup 
    \Big(I^1_0 \times \prod\limits_{i\in [d]\setminus \{1\}} I^i\Big)\,
  \end{split}  
\end{equation*}
and to show that
\begin{equation}\label{eq4}
  K_r(f,t^r)_{p,Q}  \ll K_r(f,t^r)_{p,Q_1} + K_r(f,t^r)_{p,Q_0}\,.
\end{equation}
We start with an increasing function $\varphi\in C^{\infty}(\RR)$ such that 
$$
    \varphi(s) = \left\{\begin{array}{rcl}
                    0&:& s<0\\
                    1&:& s>1                  
                 \end{array}\right..
$$
Putting $h= d_1-c_1$ and 
$$
  \lambda(s) = \varphi\Big(\frac{s-c_1}{h}\Big)\quad,\quad s\in \RR\,,
$$  
we obtain a $C^{\infty}(\RR)$-function $\lambda$ that equals zero on $[a_1,c_1]$, 
equals one on on $[d_1,b_1]$, and is increasing on $[c_1,d_1]$\,. As a direct consequence we get
\begin{equation*} 
\|\varphi^{(k)}\|_{\infty, \RR} 
\ \le \ 
h^{-k}\|\varphi^{(k)}\|_{\infty,\RR}\quad, \quad k \in \NN.
\end{equation*}
Let now $f \in W^r_p(Q)$ and $t \in {\RR}^2_+$ with $t_i \le d_i - c_i, \ i \in [d]$.
For arbitrary $g_1 \in W^r_p(Q_1)$ and $g_0 \in W^r_p(Q_0)$, put 
\begin{equation*}
 \begin{split}
    g(x) &= \lambda(x_1)g_0(x) + (1 - \lambda(x_1))g_1(x)\\
    &= g_1(x) + \lambda(x_1)(g_0(x)-g_1(x))\,.
 \end{split}
\end{equation*}
First of all, the function $g$ is defined on $Q_0 \cap Q_1 \subset Q$. 
We extend $g$ by $g_0$ on $Q_0\setminus Q_1$ and by $g_1$ on $Q_1 \setminus Q_0$ and denote the 
result also by $g$. By the construction of 
$\lambda$ this $g$ belongs to $W^r_p(Q)$ and we have
\begin{equation} \label{Ineq:[|f-g|]}
\begin{split}
    \|f - g\|_{p,Q} &\leq \|\lambda(x_1)f(x)-\lambda(x_1)g_0(x)+(1-\lambda(x_1))f(x)-(1-\lambda(x_1))g_1(x)\|_{p,Q}\\
    & \leq  \|f - g_0\|_{p,Q_0} + \|f - g_1\|_{p,Q_1}.
\end{split}
\end{equation} 
Furthermore, for any non-empty fixed subset $e \subset [d]$ we have
$$
    g^{(r(e))}(x) = g_1^{(r(e))}(x) + \sum\limits_{k=0}^{r_1} 
\binom{r_1}{k} \lambda^{(r_1-k)}(x_1)(g^{(k,\tilde{r}(e))}_1(x)-g^{(k,\tilde{r}(e))}_0(x) )
$$ 
on $Q_0 \cap Q_1$, where $\tilde{r}(e)$ denotes the vector $r(e\setminus\{1\})$ without the leading $0$. Note that
$\tilde{r}(e)$ might be zero, which happens in case $e = \{1\}$. Consequently, 
\begin{equation}\label{eq3}
 \begin{split}
   &\Big(\prod\limits_{i\in e}t_i^{r_i}\Big)\|g^{(r(e))}\|_{p,Q_0\cap Q_1}\\  
   &~~\ll \Big(\prod\limits_{i\in e}t_i^{r_i}\Big) 
   \Big(\|g_1^{(r(e))}\|_{p,Q_0\cap Q_1}
	+\max\limits_{0\leq k\leq r_1} h^{-(r_1-k)}\|g^{(k,\tilde{r}(e))}_1-g^{(k,\tilde{r}(e))}_0\|_{p,Q_1\cap Q_0}\Big)\\
   &~~\ll \Big(\prod\limits_{i\in e\setminus\{1\}}t_i^{r_i}\Big) 
   \Big(t_1^{r_1\chi_e(1)}\|g_1^{(r(e))}\|_{p,Q_0\cap Q_1}+\max\limits_{0\leq k\leq r_1} 
   \Big(\frac{t_1}{h}\Big)^{r_1-k}t_1^k\|g^{(k,\tilde{r}(e))}_1-g^{(k,\tilde{r}(e))}_0\|_{p,Q_1\cap Q_0}\Big)\,.
 \end{split}
\end{equation}
We apply the first relation in Lemma \ref{Lemma: DerivativeIneq} in the first component of 
$g_1^{(k,\tilde{r}(e))} - g_0^{(k,\tilde{r}(e))}$ using a similar argument as after 
(\ref{eq1}) and get
$$
t_1^{k} \|g_1^{(k,\tilde{r}(e))} - g_0^{(k,\tilde{r}(e))}\|_{p,Q_0\cap Q_1} 
\ll  
\|g^{(0,\tilde{r}(e))}_1 - g^{(0,\tilde{r}(e))}_0\|_{p,Q_0\cap Q_1} 
+ t_1^{r_1}\|g_1^{(r(e))} - g_0^{(r(e))}\|_{p,Q_0\cap Q_1}.  
$$
Plugging this into \eqref{eq3} and taking $t_1 \leq h$ into account gives in case $\tilde{r}(e) \neq 0$
\begin{equation}\label{eq2}
  \begin{split}
       \Big(\prod\limits_{i\in e}t_i^{r_i}\Big)\|g^{(r(e))}\|_{p,Q_0\cap Q_1}
       \ll &\Big(\prod\limits_{i\in e}t_i^{r_i}\Big)\|g_0^{(r(e))}\|_{p,Q_0}+
       \Big(\prod\limits_{i\in e\setminus \{1\}}t_i^{r_i}\Big)\|g_0^{(0,\tilde{r}(e))}\|_{p,Q_0}\\
       &+\Big(\prod\limits_{i\in e}t_i^{r_i}\Big)\|g_1^{(r(e))}\|_{p,Q_1}+
       \Big(\prod\limits_{i\in e\setminus \{1\}}t_i^{r_i}\Big)\|g_1^{(0,\tilde{r}(e))}\|_{p,Q_1},
  \end{split}
\end{equation}
and in case $\tilde{r}(e) = 0$, i.e., if $e = \{1\}$, 
\begin{equation}\label{eq7}
  \begin{split}
       \Big(\prod\limits_{i\in e}t_i^{r_i}\Big)\|g^{(r(e))}\|_{p,Q_0 \cap Q_1}
       \ll& \Big(\prod\limits_{i\in e}t_i^{r_i}\Big)\|g_0^{(r(e))}\|_{p,Q_0}+
       \|f-g_0\|_{p,Q_0}\\
       &+\Big(\prod\limits_{i\in e}t_i^{r_i}\Big)\|g_1^{(r(e))}\|_{p,Q_1}+
       \|f-g_1\|_{p,Q_1}\,.
  \end{split}
\end{equation}
Using that 
$$
   \|g^{(r(e))}\|_{p,Q} \leq \|g^{(r(e))}\|_{p,Q_0\cap Q_1} + \|g_0^{(r(e))}\|_{p,Q_0} + \|g_1^{(r(e))}\|_{p,Q_1},
$$
we obtain together with \eqref{Ineq:[|f-g|]}, \eqref{eq2}, and \eqref{eq7} the relation
$$
    K_r(f,t^r)_{p,Q} \ll K_r(f,t^r)_{p,Q_{0}} + K_r(f,t^r)_{p,Q_1}
$$
which is \eqref{eq4}\,. We continue with the same procedure, this time with $Q_1$ and 
$Q_0$ instead of $Q$, proving that (analogously for $Q_1$)
$$
    K_r(f,t^r)_{p,Q_0} \ll K_r(f,t^r)_{p,Q_{01}} + K_r(f,t^r)_{p,Q_{00}},
$$
where 
$$
   Q_{00} = \Big(I^1_0 \times I^2_0 \times  
\prod\limits_{i\in [d]\setminus \{1,2\}} I^i\Big)\quad\mbox{and}\quad Q_{01} = \Big(I^1_0 \times I^2_1 \times  
\prod\limits_{i\in [d]\setminus \{1,2\}} I^i\Big),
$$
and so forth. An iteration of this argument finishes the proof. 
\end{proof}
\subsection{Proof of Theorem \ref{Theorem: K_r><omega_r}}
\begin{proof}
The first inequality in \eqref{ineq: K_r><omega_r} follows from the definition. 
Namely, if $f \in L_p(Q)$, 
for any non-empty $e \subset [d]$ and any $g \in W^r_p(Q)$, we have
\begin{equation*} 
\begin{aligned}
\omega_{r(e)}(f,t)_{p,Q}
\ & \le \ 
\omega_{r(e)}(f - g,t)_{p,Q} + \omega_{r(e)}(g,t)_{p,Q} \\
\ & \le \ 
\Big(\prod_{i \in e} 2^{r_i}\Big) \left\{ \|f - g\|_{p,Q} 
+ \Big(\prod_{i \in e} t_i^{r_i}\Big) \|g^{(r(e))}\|_{p,Q}\right\}.
\end{aligned}
\end{equation*}
Hence, we obtain the first inequality in \eqref{ineq: K_r><omega_r}. Let us prove the second one.
For simplicity we prove it for $d=2$ and $t \in {\RR}^2_+,\ t>0$. 
If $k$ is a natural number, then we define for univariate functions $\varphi$ on the interval 
$[a,b]$ the operator $P^k_t, t \ge 0,$ by
\begin{equation*}
P^k_t(\varphi, x):= \varphi(x) + (-1)^{k + 1} \int_{-\infty}^\infty \Delta^{k}_{th}(\varphi, x) M_k(h) dh\,,
\end{equation*}
where $M_k$ is the B-spline of order $k$ with knots at the integer points 
$0,...,k$, and support $[0,k]$. The function $P^k_t(\varphi)$ is defined 
on $[a,b-h/4]$ for $t\le {\bar t}:=h/4k^2$, where $h:= b-a$.
We have, see \cite[page 177]{DL},
\begin{equation} \label{eq:[{P^k_t(varphi)}^{(k)}]}
\{P^k_t(\varphi)\}^{(k)}(x)= t^{-k} \sum_{j=1}^k(-1)^{j + 1} j^{-k}\Delta^k_{jt}(\varphi, x).
\end{equation}
Put $h_i:= b_i-a_i$ and $c_i := a_i + h_i/4,\ d_i := b_i - h_i/4, i \in[2]$. It holds
$a_i < c_i < d_i < b_i$, and we will use the notation $Q_e$ given in \eqref{def:[Q_e]} for any $e \subset [d]$. 
In particular, we have $Q_{[2]}=[a_1,d_1] \times [a_2,d_2]$. 
For functions $f$ on the parallelepiped $Q=[a_1,b_1]\times[a_2,b_2]$ the operator
$P^r_t, t \in {\RR}^2_+$, is defined by
\begin{equation*}  
 P^r_t(f) := \prod_{i=1}^2 P^{r_i}_{t_i,i}(f),
\end{equation*}
where the univariate operator $P^{r_i}_{t_i,i}$ is
applied to the univariate function $f$ by considering $f$ as a 
function of  variable $x_i$ with the remaining variables fixed. 
The function $P^k_t(f)$ is defined on $Q_{[2]}$
for $t\le {\bar t}$, where ${\bar t}_i:=h_i/4r_i^2$.
We have
\begin{equation*}  
\begin{aligned}
P^r_t(f,x)  
\ &  = \ 
f(x) +   (-1)^{r_1 + 1} \int_{-\infty}^\infty \Delta^{r^1}_{th}(f, x) M_{r_1}(h_1) dh_1
\ + \ (-1)^{r_2 + 1} \int_{-\infty}^\infty \Delta^{r^2}_{th}(f, x) M_{r_2}(h_2) dh_2 \\
\ &  + \ 
(-1)^{r_1 + r_2 + 2} \int_{-\infty}^\infty \int_{-\infty}^\infty 
\Delta^r_{th}(f, x) M_{r_1}(h_1)M_{r_2}(h_2) dh_1dh_2, 
\end{aligned}
\end{equation*}
where $r^1 := (r_1,0)$ and $r^2:= (0,r_2)$.
Let us define the function $g_t=P^r_t(f)$. If $f \in L_p(Q)$, by  Minkowski's inequality
and properties of the B-spline $M_{r_i}$ we get
\begin{equation} \label{ineq:Norm[f-g_t(f)]} 
\|f - g_t\|_{p,Q_{[2]}}  
\ \ll \ 
\omega_{r^1}(f,t)_{p, Q_{[2]}} \ + \ \omega_{r^2}(f,t)_{p, Q_{[2]}} \ + \ \omega_r(f,t)_{p, Q_{[2]}} 
 \ = \ \Omega_r(f,t)_{p, Q_{[2]}}. 
\end{equation}
Further, by \eqref{eq:[{P^k_t(varphi)}^{(k)}]} we obtain
\begin{equation*} 
g_t^{(r^1)} 
\ = \ 
 P^{r_2}_{t_2,2}(\{P^{r_1}_{t_1,1}(f)\}^{(r^1)})
\ = \  
P^{r_2}_{t_2,2}\Big\{t_1^{-r_1} \sum_{j_1=1}^{r_1}(-1)^{j_1 + 1} j_1^{-r_1}
\binom{r_1}{j_1}\Delta^{r_1}_{j_1t_1,1}(f)\Big\}.
\end{equation*}
Since $P^{r_2}_{t_2,2}$ is a linear bounded operator from $L_p(Q_{[2]})$ into $L_p(Q_{[2]})$ and further\\ 
$\|\Delta^{r_1}_{j_1t_1,1}(f)\|_{p,Q_{[2]}} \ll \omega_{r^1}(f,t)_{p, Q_{[2]}}$, we have
\begin{equation} \label{ineq:Norm[g_t^{(r^1)}]} 
t_1^{r_1} \|g_t^{(r^1)} \|_{p,Q_{[2]}}  
\ \ll \ 
\omega_{r^1}(f,t)_{p, Q_{[2]}}. 
\end{equation}
Similarly, we can prove that
$$ 
t_2^{r_2} \|g_t^{(r^2)} \|_{p,Q_{[2]}}  
\ \ll \ 
\omega_{r^2}(f,t)_{p, Q_{[2]}}. 
$$
Again, by \eqref{eq:[{P^k_t(varphi)}^{(k)}]} we get
$$
g_t^{(r)} 
\ = \ 
 t_1^{-r_1} t_2^{-r_2}\sum_{j_1=1}^{r_1}\sum_{j_2=1}^{r_2}
(-1)^{j_1 + j_2 + 2}j_1^{-r_1}j_2^{-r_2}\binom{r_1}{j_1}\binom{r_2}{j_2}\Delta^r_{jt}(f).
$$
From the inequality 
$\|\Delta^r_{jt}(f)\|_{p,Q} \ll \omega_r(f,t)_{p, Q_{[2]}}$ it follows that
\begin{equation} \label{ineq:Norm[g_t^{(r)}]} 
t_1^{r_1} t_2^{r_2}\|g_t^{(r)} \|_{p,Q_{[2]}}  
\ \ll \ 
\omega_r(f,t)_{p, Q_{[2]}}. 
\end{equation} 
Combining \eqref{ineq:Norm[f-g_t(f)]},\eqref{ineq:Norm[g_t^{(r^1)}]}--\eqref{ineq:Norm[g_t^{(r)}]} gives
\begin{equation*} 
\|f - g_t\|_{p,Q_{[2]}} \ + \ t_1^{r_1} \|g_t^{(r^1)} \|_{p,Q_{[2]}}  
\ + \ t_2^{r_2} \|g_t^{(r^2)} \|_{p,Q_{[2]}} \ + \ t_1^{r_1} t_2^{r_2}\|g_t^{(r)} \|_{p,Q_{[2]}} 
\ \ll \ 
 \Omega_r(f,t)_{p, Q}. 
\end{equation*}
Therefore, we get 
\begin{equation*} 
  K_r(f,t^r)_{p,Q_{[2]}} \ \ll \ \Omega_r(f,t)_{p,Q}\,,
\end{equation*}
and in a similar way
\begin{equation*} 
  K_r(f,t^r)_{p,Q_e} \ \ll \ \Omega_r(f,t)_{p,Q} 
\end{equation*}
for any subset $e \subset [2]$, where $Q_e$ is given by (\ref{def:[Q_e]}).
The last inequality and Lemma \ref{Lemma: Ineq[K_r]} prove \eqref{ineq: K_r><omega_r} for $t \le {\bar t}$.
Now take a function ${\bar g}\in W^r_p(Q)$ such that
\begin{equation}\label{k<omega} 
\|f - {\bar g}\|_{p,Q} \ + \ {\bar t}_1^{r_1} \|{\bar g}^{(r^1)} \|_{p,Q}  
 \ + \ {\bar t}_2^{r_2} \|{\bar g}^{(r^2)} \|_{p,Q} 
\ + \ {\bar t}_1^{r_1} {\bar t}_2^{r_2}\|{\bar g}^{(r)} \|_{p,Q} 
\ \ll \ 
 \Omega_r(f,{\bar t})_{p,Q}. 
\end{equation}
By Theorem \ref{Theorem:UpperBnd[E_{r-1}(f)]} we have
\begin{equation}\label{taylorappr} 
\|{\bar g} - T_r({\bar g})\|_{p,Q}
 \ \ll \ {\bar t}_1^{r_1} \|{\bar g}^{(r^1)} \|_{p,Q}  
 \ + \ {\bar t}_2^{r_2} \|{\bar g}^{(r^2)} \|_{p,Q} 
\ + \ {\bar t}_1^{r_1} {\bar t}_2^{r_2}\|{\bar g}^{(r)} \|_{p,Q}\\ \,. 
\end{equation}
Since $T_r ({\bar g}) \in W^r_p(Q)$ and $(T_r ({\bar g}))^{(r(e))} = 0$ for every non-empty subset $e\subset [d]$, 
it holds for all  $t > {\bar t}$
\begin{equation*} 
\begin{split}
  K_r(f,t^r) &\leq \|f-T_r({\bar g})\|_{p,Q}\\
  &\leq \|f-\bar{g}\|_{p,Q} + \|\bar{g}-T_r({\bar g})\|_{p,Q}\\
  &\ll 
  \Omega_r(f,{\bar t})_{p,Q}\leq \Omega_r(f,t)_{p,Q}\,,
\end{split}
\end{equation*}
where the third step combines \eqref{k<omega} and \eqref{taylorappr}.
Therefore, \eqref{ineq: K_r><omega_r} has been proved for arbitrary $t > 0.$
\end{proof}

\section{ Whitney type inequalities}
\label{Whitney's inequality}

Using the results from Section \ref{Johnen's inequality} we are now able to prove
Theorem \ref{Theorem: Whitney}.

\begin{proof}
The first inequality in \eqref{ineq:WhitneyIneq} is trivial.
Indeed, if $f \in L_p(Q)$ then for any non-empty $e \subset [d]$ and any $\varphi \in \Pp_r$ we have
\begin{equation*} 
\begin{aligned}
\omega_{r(e)}(f,\delta)_{p,Q}
\ & = \ 
\omega_{r(e)}(f - \varphi, \delta)_{p,Q} \\
\ & \le \ 
\Big(\prod_{i \in e} 2^{r_i}\Big)  \|f - \varphi\|_{p,Q}.
\end{aligned}
\end{equation*}
Hence, we obtain the first inequality in \eqref{ineq:WhitneyIneq}.
On the other hand, from Theorem \ref{Theorem:UpperBnd[E_{r-1}(f)]} it follows that 
for any $g \in W^r_p(Q)$
\begin{equation*} 
\begin{aligned}
E_r(f)_{p,Q}  
\ & \le \ 
\|f - g\|_{p,Q} \ + \ E_r(g)_{p,Q} \\
 \ & \le \ 
\|f - g\|_{p,Q} \ + \
\|g - T_r(g)\|_{p,Q} \\ 
\ & \ll \ 
\|f - g\|_{p,Q} 
\ + \ 
\Big(\prod_{i \in e} \delta_i^{r_i}\Big) \|g^{(r(e))}\|_{p,Q}.
\end{aligned}
\end{equation*}
Hence, we get 
$$
 E_r(f)_{p,Q}
 \  \ll \ 
K_r(f,\delta^r)_{p,Q}.
$$
By Theorem \ref{Theorem: K_r><omega_r} we have proved the second inequality in \eqref{ineq:WhitneyIneq}.   
\end{proof}

The result in Theorem \ref{Theorem: Whitney} can be slightly modified. 
For $r \in {\ZZ}^d_+$, $h\in \RR^d$, $e \subset [d]$ and a 
$d$-variate function $f:\RR^d \to \RR$ 
the mixed $p$-mean modulus of smoothness of order $r(e)$ is given by
\begin{equation*}
w_{r(e)}(f,t)_p 
\ := \ 
\left( \Big(\prod_{i \in e} t_i^{- 1}\Big) \int_{U(t)} \int_{Q_{r(e)h}}
|\Delta_h^{r(e)}(f,x)|^p \ dx \ dh \right)^{1/p}, \ t \in \RR^d_+,
\end{equation*}
where $U(t):= \{ h \in {\RR}^d: |h_i| \le t_i, \ i \in [d]\}$, 
with the usual change of the outer mean integral to sup if $p= \infty$. This leads to 
the definition of the {\em total mixed $p$-mean modulus of smoothness of order $r$} by
\begin{equation*} 
W_r(f,t)_{p,Q}
:= \ 
 \sum_{e \subset [d], e \ne \emptyset} w_{r(e)}(f,t)_{p,Q}, \quad t \in {\RR}^d_+.
\end{equation*}
Note that $W_r(f,t)_{p,Q}$ coincides with $\Omega_r(f,t)_{p,Q}$ when $p= \infty$.
In a way similar to the proof of Theorem \ref{Theorem: Whitney} we can prove the following slightly stronger result.

\begin{theorem} 
Let $1 \le p \le \infty$, $r \in {\NN}^d$. Then there are constants $C, C'$ depending 
only on $r,d$ such that for every $f \in L_p(Q)$
\begin{equation*}  
C W_r(f,\delta)_{p,Q} 
\ \le \ 
E_r(f)_{p,Q}  
\ \le \ 
C' W_r(f,\delta)_{p,Q} 
\end{equation*}
where $\delta = \delta(Q)$. 
\end{theorem}

\begin{remark} A corresponding inequality in the case $0<p<1$ is so far left open for subsequent contributions. 
It seems that the modulus $W_r(f,t)_{p,Q}$ is suitable to treat this case, cf. the appendix of \cite{HeNe07}.  
\end{remark}

{\bf Acknowledgements}
\ Since main parts of the paper have been worked out during a stay of the second named author 
at the Information Technology Institute of the Vietnam National University, Hanoi in 2010, 
Tino Ullrich would like to thank here Prof. Dinh Dung for his kind invitation and warm hospitality. 
He would also like to thank the Hausdorff-Center for Mathematics at University of Bonn (Germany) and 
Prof. Holger Rauhut for giving additional financial support. The work of the first named author was supported by the 
National Foundation for Development of Science and Technology (Vietnam). Both authors would also like to thank Dany
Leviatan for a critical reading of the manuscript and for several valuable comments how to improve it.


\begin{thebibliography}{99}

\bibitem{Br64} Yu.A. Brudny\u{\i}, 
On a theorem on best local approximations,  
{\em Kazanskii Gosudarstvennyi Universitet. Uchenye Zapiski}
{\bf 124}, No. 6 (1964), 43--49.

\bibitem{Br70} Yu.A. Brudny\u{\i}, 
A multidimensional analogue of a certain theorem of Whitney, 
{\em Math. USSR-Sb. }
{\bf 2} (1970), 157--170.

\bibitem{Br70_2} Yu.A. Brudny\u{\i}, 
Approximation of functions of $n$ variables by
quasipolynomials,  
{\em Izv. Akad. Nauk SSSR Ser. Mat.}
{\bf 34}, No. 3 (1970), 564--583.

\bibitem{De10}S. Dekel, 
On the equivalence of the modulus of smoothness and the K-functional over convex domains, 
{\em Journ. Approx. Theory}
 {\bf 162} (2010), 349--362.  

\bibitem{De04}S. Dekel and D. Leviatan, 
Whitney estimates for convex domains with applications to multivariate piecewise polynomial approximation, 
{\em Found. Comput. Math.} {\bf 4} (2004), 345--368.

\bibitem{DL} R.A. DeVore and G.G. Lorentz, 
Constructive approximation, 
Springer-Verlag, New York, 1993.

\bibitem{DePeTe94} R.A. DeVore, P.P. Petrushev, and V.N. Temlyakov,
 Multivariate trigonometric polynomial approximations with frequencies
from the hyperbolic cross, {\em Mat. Zametki} {\bf 56}, No. 3 (1994), 36--63.

\bibitem{Di86} Dinh Dung,
Approximation of functions of several variables on tori by trigonometric polynomials, 
{\em Mat. Sb.} {\bf 131}(1986), 251-271.

\bibitem{HeNe07}
L.~I. {H}edberg and Yu.~{N}etrusov,
An axiomatic approach to function spaces, spectral
  synthesis, and Luzin approximation,
{\em Mem. Amer. Math. Soc.} {\bf 188}, No. 882 (2007).

\bibitem{GaTo02} G. Garrig\'{o}s and A. Tabacco, 
Wavelet decompositions of anisotropic Besov spaces, 
{\em Math. Nachr.} 
{\bf 239--240} (2002), 80--102.

\bibitem{J} H. Johnen,
Inequalities connected with moduli of smoothness,
{\em Math. Vesnik} {\bf 3} (1972), 289--303.

\bibitem{JS76} H. Johnen and K. Scherer,
On the equivalence of the K-functional and moduli of continuity and some applications, in: 
{\em Constructive Theory of Functions of Several Variables}, 
Proc. Conf. Math. Res. Inst., Oberwolfach, 1976, 119--140; 
Lecture Notes {\bf 571}, Springer Berlin, 1977. 

\bibitem{P} J. Peetre,
A theory of interpolation of normed spaces,
Notes Universidade de Brasilia, 1963. 

\bibitem{Sp74}G. Sparr, 
Interpolation of several Banach spaces, 
{\em Ann. Mat. Pur. Appl.} 
{\bf 99}, No. 1 (1974), 247-316.

\bibitem{S} \`{E}.A. Storozhenko,
Approximation by algebraic polynomials of functions in the class $L^p, \ 0 < p <1$,
{\em Izv. Akad. Nauk SSSR, Ser. Mat. }{\bf 41} (1977), 652--662.

\bibitem{StOs78} \`{E}.A. Storozhenko and P. Oswald, Jackson's theorem in the spaces $L^p(\RR^k), 0 < p < 1$, 
{\em Siberian Math. J.} {\bf 19} (1978), 630-640.

\bibitem{ST87}
H.-J. Schmeisser and H. Triebel,  Topics in Fourier analysis
and function spaces, Wiley, Chichester, 1987.

\bibitem{Ull06}
T.~Ullrich, Function spaces with dominating mixed smoothness,
characterization by differences, {\em Jenaer Schriften zur Mathematik und
Informatik}, Math/Inf/05/06 (2006).

\bibitem{Vy06} J.~Vybiral, 
Function spaces with dominating mixed
smoothness, 
{\em Dissertationes Math.} 
{\bf 436} (2006), 73~pp.

\bibitem{Wh57} H. Whitney,
On functions with bounded $n$th difference,
{\em J. Math. Pures Appl.} 
{\bf 36} (1957), 67--95.

\end{thebibliography}
\end{document}